\documentclass[leqno,11pt]{scrartcl}
\usepackage[utf8]{inputenc}

\usepackage{amsmath,amssymb,amsthm}
\usepackage{amscd}
\usepackage{setspace} 

\usepackage{enumerate}
\usepackage{array}
\usepackage{tabularx}
\usepackage{multirow}
\usepackage{booktabs}

\theoremstyle{definition}
\newtheorem{definition}{Definition}

\newtheorem*{remarkson}{Remarks}

\theoremstyle{plain}

\newtheorem{theorem}{Theorem}

\newtheorem{lemma}[definition]{Lemma}

\newtheorem{corollary}{Corollary}

\theoremstyle{remark}

\newcommand{\C}{\mathbb{C}}

\newcommand{\G}{\mathbb{G}}
\renewcommand{\H}{\mathbb{H}}
\newcommand{\N}{\mathbb{N}}
\newcommand{\Pb}{\mathbb{P}}
\newcommand{\Q}{\mathbb{Q}}
\newcommand{\R}{\mathbb{R}}
\newcommand{\Z}{\mathbb{Z}}

\newcommand{\Hcal}{\mathcal{H}}
\newcommand{\Mcal}{\mathcal{M}}
\newcommand{\Ocal}{\mathcal{O}}

\newcommand{\Tcal}{\mathcal{T}}

\newcommand{\trace}{\operatorname{trace}\,}
\newcommand{\diag}{\operatorname{diag}\,}

\newcommand{\Sp}{\operatorname{Sp}}
\newcommand{\im}{\operatorname{Im}}

\let\leq\leqslant
\let\geq\geqslant

\begin{document}

\begin{center}
\begin{huge}
\begin{spacing}{1.0}
\textbf{The Maaß Space for Paramodular Groups}  
\end{spacing}
\end{huge}

\bigskip
by
\bigskip

\begin{large}
\textbf{Bernhard Heim\footnote{Bernhard Heim, GUtech, Way No. 36, Building No. 331, North Ghubrah, Muscat, Sultanate of Oman, bernhard.heim@gutech.edu.om}} and 
\textbf{Aloys Krieg\footnote{Aloys Krieg, Lehrstuhl A für Mathematik, RWTH Aachen University, D-52056 Aachen, krieg@rwth-aachen.de}}
\end{large}
\vspace{0.5cm}\\
12 March 2018
\vspace{1cm}
\end{center}
\begin{abstract}
In this paper we describe a characterization for the Maaß space associated with the paramodular group of degree $2$ and squarefree level $N$. As an application we show that the Maaß space is invariant under all Hecke operators. As a consequence we conclude that the associated Siegel-Eisenstein series belongs to the Maaß space. 
\end{abstract}
\vspace{2ex}

\newpage
\section{Introduction}

Paramodular forms have first been studied by Siegel \cite{Si} in the sixties. Recently paramodular forms of degree two and squarefree level $N$ came into focus again. For example
\begin{itemize}
 \item The paramodular conjecture by Brumer and Kramer \cite{BK}
 \item New and old form theory by Roberts and Schmidt \cite{RS2}
 \item The Böcherer conjecture for paramodular forms \cite{RT}
 \item Borcherds lifts \cite{Bo}
\end{itemize}
Most of the results have in common that a certain subspace of generalized Saito-Kurokawa lifts, the Maaß space, plays a significant role. Although several authors have studied these lifts in the orthogonal setting (cf. Sugano \cite{Su} or Oda \cite{O} or \cite{HM1}) or as a direct generalization of the original approach of Maaß by Gritsenko (\cite{G1}, \cite{G3}, \cite{G2}), not all results are stated in an explicit form such that they can easily be applied in the works above and that certain refined properties had not been worked out.

In this paper we give a characterization of the Maaß space for the paramodular group similar to \cite{H} in the case of the Siegel modular group. This characterization is applied in order to show that the Maaß space is invariant under the full Hecke algebra described in \cite{GK}. As an application we can show that the Siegel-Eisenstein series always belongs to the Maaß space.

\section{The Maaß space}

Let 
\[
 \H_2:=\left\{Z=\begin{pmatrix}
                 \tau & z  \\  z & \tau'
                \end{pmatrix}
\in \C^{2\times 2};\;Z=Z^{tr},\,\im Z>0\right\}
\]
denote the Siegel half-space of degree $2$ and
\begin{align*}
 & \Sp_2(\R):=\{M\in\R^{4\times 4};\, J[M]:=M^{tr}JM=J\},\\[1ex]
 & J= \begin{pmatrix}
 0 & -I  \\  I & 0
\end{pmatrix}, \;\;
I=\begin{pmatrix}
   1 & 0  \\  0 & 1
  \end{pmatrix}
\end{align*}
the symplectic group of degree $2$. $\Sp_2(\R)$ acts on $\H_2$ in the usual way
\[
 Z\mapsto M\langle Z\rangle:=(AZ+B)(CZ+D)^{-1},\;\; M=\begin{pmatrix}
                                                       A & B  \\  C & D
                                                      \end{pmatrix},
\]
and on the set of holomorphic functions $f:\H_2\to \C$ for $k\in\Z$ via
\[
 f(Z)\mapsto f \underset{k}{\mid} M(Z):=\det (CZ+D)^{-k} f(M\langle Z\rangle). 
\]
Given $N\in\N$ let 
\[
 \Sigma_N:=\left\{\begin{pmatrix}
                   * & *N & * & *  \\ * & * & * & */N  \\ * & *N & * & *  \\  *N & *N & *N & *
                  \end{pmatrix}
\in \Sp_2(\Q);\;*\in\Z\right\}
\]
be the \textit{paramodular group of degree $2$ and level $N$}. A modular form $f\in \Mcal_k(\Sigma_N)$ is a holomorphic function $f:\H_2\to\C$ satisfying
\[
 f \underset{k}{\mid} M = f\;\;\text{for all}\;\;  M\in\Sigma_N
\]
and possesses a Fourier expansion of the form 
\[
  f(Z) = \sum_{T\geq 0} \alpha_f(T)\, e^{2\pi i\trace (TZ)}, \quad T = \begin{pmatrix}
                                                                n & r/2 \\  r/2 & mN
                                                               \end{pmatrix},\;\;
 \begin{matrix}
  n,m\in\N_0, \;r\in\Z ,\\
  4nm N-r^2\geq 0,\;\;\;
 \end{matrix}
\]
and a Fourier Jacobi expansion of the form
\begin{gather*}\tag{1}\label{gl_1}
 f(Z) = \sum^{\infty}_{m=0} f^*_m (Z) = \sum^{\infty}_{m=0} f_{mN}(\tau,z)\, e^{2\pi imN\tau'}, \quad Z = \begin{pmatrix}
                                                                \tau & z \\  z & \tau'
                                                               \end{pmatrix} \in\H_2,
\end{gather*}
where
\begin{gather*}
 f_{mN}(\tau,z) = \sum_{n,r} \alpha_f(T)\, e^{2\pi i(n\tau+rz)}.
\end{gather*}
We will refer to both $f_{mN}$ and $f^*_{mN}$ as Jacobi forms. Next we define the Maaß space $\Mcal^*_k(\Sigma_N)$ to consist of all $f\in \Mcal_k(\Sigma_N)$ such that for all $T\neq 0$
\begin{gather*}\tag{2}\label{gl_2}
\alpha_f(T) = \sum_{\delta\,|\,gcd(n,r,m)} \delta^{k-1} \alpha_f(nm/\delta^2,r/\delta,N),
\end{gather*}
where $\alpha_f(n,r,mN)$ denotes the Fourier coefficient of $\left(\begin{smallmatrix}
                                                                   n & r/2 \\  r/2 & mN 
                                                                   \end{smallmatrix}\right)$. 
 
If $J_{k,N}$ denotes the space of Jacobi forms of weight $k$ and index $N$ in the sense of Eichler-Zagier \cite{EZ}, Gritsenko \cite{G1} demonstrated that the map
\[                                                                                                                            \varphi: \Mcal^*_k (\Sigma_N) \to J_{k,N}, \quad f\mapsto f_N,
\]
is an isomorphism, where we have for even $k\geq 4$
\begin{gather*}\tag{3}\label{gl_3}
f_0(\tau,z) = f^*_0(Z) = \alpha_f(0,0,0) E_k(\tau), \quad \alpha_f(0,0,N) = -\frac{2k}{B_k}\,\alpha_f(0,0,0)
\end{gather*}
with the normalized elliptic Eisenstein series
\[
E_k(\tau) = 1-\frac{2k}{B_k} \sum^{\infty}_{n=1} \sigma_{k-1}(n)\, e^{2\pi in\tau}.
\]
Consider the embedding
\begin{gather*}\tag{4}\label{gl_4}
\begin{split}
 SL_2(\R) \times SL_2(\R) & \hookrightarrow \Sp_2(\R), \\[0.5ex]
\begin{pmatrix}
        a & b  \\  c & d
       \end{pmatrix} \times 
      \begin{pmatrix}
        a' & b'  \\  c' & d'
       \end{pmatrix} & : =
 \begin{pmatrix}
  a & 0 & b & 0 \\0 & a' & 0 & b'/N  \\  c & 0 & d & 0  \\  0 & c'N & 0 & d'
 \end{pmatrix}.
 \end{split}
\end{gather*}
Let 
\[
 \Mcal(n):=\left\{\frac{1}{\sqrt{n}} M\in SL_2(\R);\;M\in\Z^{2\times 2}\right\}, \;\; n\in\N.
\]
Here $\Gamma = SL_2(\Z)=\Mcal(1)$.

We define for $f\in \Mcal_k(\Sigma_N)$
\begin{gather*}\tag{5}\label{gl_5}
\begin{split}
f \underset{k}{\mid} \Tcal_n^{\uparrow}: & = \sum_{M:\Gamma\backslash \Mcal(n)} f\underset{k}{\mid} M\times I , \\
f \underset{k}{\mid} \Tcal_n^{\downarrow}: & = \sum_{M:\Gamma\backslash \Mcal(n)} f\underset{k}{\mid} I\times M .
\end{split}
\end{gather*}
Note that \eqref{gl_4} and \eqref{gl_5} are not Hecke operators on $\Mcal_k(\Sigma_N)$. Therefore we need the so called \textit{Jacobi group} inside $\Sigma_N$ given by
\[
 \Sigma_N^J:=\left\{\begin{pmatrix}
                     * & 0 & * & * \\ * & \pm 1 & * & */N \\ * & 0 & * & * \\ 0 & 0 & 0 & \pm 1
                    \end{pmatrix}
\in \Sigma_N\right\}.
\]
Then the elements $f_{mN} \in J_{k,m}$ can be characterized by
\[
 f^*_{mN} \underset{k}{\mid} M=f^*_{mN} \quad \text{for all}\quad M\in \Sigma_N^J.
\]
We define Hecke operators just as in \cite{F} without any normalizing factors by
\[
f \underset{k}{\mid}\Gamma M\Gamma:=\sum_{L:\Gamma\backslash \Gamma M\Gamma}  f \underset{k}{\mid} L,
\]
whenever $f$ is invariant under the subgroup $\Gamma$ of $\Sp_2(\R)$ and if $M\in\Sp_2(\R)$ such that the sum is finite.

\begin{lemma}\label{lemma 1} 
Let $f\in \Mcal_k(\Sigma_N)$, $N\in \N$ and $p\in \Pb$ be a prime. Then the following assertions are equivalent
\begin{flalign*}
(i) \quad & f \underset{k}{\mid} \Tcal_p^{\uparrow} = f \underset{k}{\mid} \Tcal_p^{\downarrow}. & &\\
(ii)\quad & \alpha_f(n,r,pmN) + p^{k-1} \alpha_f(n,\tfrac{r}{p},\tfrac{m}{p}N)  & &\\ 
	  \quad &  = \alpha_f(pn,r,mN) + p^{k-1} \alpha_f(\tfrac{n}{p},\tfrac{r}{p},mN)\;\;\text{for all}\;\; n,r,m.         & &\\ 
(iii)\quad & \alpha_f(p^{\nu}n,p^{\rho}r,p^{\mu}mN) = \sum^{\min(\nu,\rho,\mu)}_{\delta=0} p^{\delta(k-1)} \alpha_f(p^{\nu+\mu-2\delta}n,p^{\rho-\delta}r,mN) ,    & &\\
	  \quad &  \alpha_f(p^{\nu}n,0,p^{\mu}mN) = \sum^{\min(\nu,\mu)}_{\delta=0} p^{\delta(k-1)}     		 \alpha_f(p^{\nu+\mu-2\delta}n,0,mN),      & &\\
	  \quad &  \alpha_f(p^{\nu}n, 0,0) = \sigma_{k-1}(p^{\nu}) \alpha_f(n,0,0),  & &\\
	  \quad & \alpha_f(0,0,p^{\mu}mN) = \sigma_{k-1}(p^{\mu})\alpha_f(0,0,mN)	   & &\\
	  \quad &  \text{for all $(n,r,m)$ such that $p\nmid n,p\nmid r, p\nmid m$} .       & &\\
(iv) \quad & f \underset{k}{\mid} \Sigma_N^J \,\tfrac{1}{\sqrt{p}}\diag(p,p,1,1)\,
   \Sigma_N^J     & &\\
     	  \quad &   = f \underset{k}{\mid} \left(\Sigma_N^J \,
     \tfrac{1}{p}\diag(p,p^2,p,1)\,\Sigma_N^J + \sum_{u\bmod{p}} \Sigma_N^J\, I\times 
      \begin{pmatrix}
       1 & u/p  \\  0 & 1                                                    
      \end{pmatrix} \Sigma_N^J\right). & &
\end{flalign*}
\end{lemma}
\noindent
Here and in the following we set $\alpha_f(n,r,mN) = 0$ if $n$ or $r$ or $m$ is not integral.

\begin{proof}
One can apply the isomorphism between $\Sigma_N$ and $\Ocal(2,3)$ (e.g. \cite{K1}) and may the use \cite{HM1} and \cite{GHK}. Alternatively we give a short direct proof. \\
(i)\;$\Leftrightarrow$\;(ii) Use the representatives 
$\frac{1}{\sqrt{p}}\left(\begin{smallmatrix}
       p & 0  \\  0 & 1                                                    
      \end{smallmatrix}\right)$, 
$\frac{1}{\sqrt{p}}\left(\begin{smallmatrix}
       1 & b  \\  0 & p                                                    
      \end{smallmatrix}\right)$, 
$b=0,\ldots,p-1$, of $\Mcal(p)$ and calculate the Fourier expansion in (i): 
\begin{align*}
 & f \underset{k}{\mid} \Tcal_p^{\uparrow}(Z) = p^{k/2} f\begin{pmatrix}
                                                          p\tau & \sqrt{p}z  \\  \sqrt{p}z & \tau'
                                                         \end{pmatrix} 
+ p^{-k/2} \sum_{u\bmod{p}} f\begin{pmatrix}
                              (\tau+u)/p & z/\sqrt{p}  \\  z/\sqrt{p} & \tau'
                             \end{pmatrix}   \\
 & = p^{-1+k/2} \sum_{n,r,m} \left(\alpha_f\left(pn,r,mN\right) + p^{k-1} \alpha_f\left(\tfrac{n}{p},\tfrac{r}{p},mN\right)\right) \,e^{2\pi i(n\tau+rz/\sqrt{p}+mN\tau')} ,  \\[1ex]
 & f \underset{k}{\mid} \Tcal_p^{\downarrow}(Z) = p^{k/2} f\begin{pmatrix}
                                                          \tau & \sqrt{p}z  \\  \sqrt{p}z & p\tau'
                                                         \end{pmatrix} 
+ p^{-k/2} \sum_{u\bmod{p}} f\begin{pmatrix}
                             \tau & z/\sqrt{p}  \\  z/\sqrt{p} & (\tau'+u/N)/p
                             \end{pmatrix}   \\
 & = p^{-1+k/2} \sum_{m,n,r} \left(\alpha_f\left(n,r,pmN\right) + p^{k-1} \alpha_f\left(n,\tfrac{r}{p},\tfrac{m}{p}N\right)\right) \,e^{2\pi i(n\tau+rz/\sqrt{p}+mN\tau')} .
\end{align*}

\noindent(ii)\;$\Rightarrow$\;(iii) The identity can be proved by a simple induction on $\mu$ using the identity (ii), where the case $\mu=0$ is clear. If $\mu\geq 1$ assume both $\nu\geq 1$ and $\rho\geq 1$. Then (ii) yields
\begin{align*}
 \alpha_f\bigl(p^{\nu}n,p^{\rho}r,p^{\mu}mN\bigr)= & \;\alpha_f \bigl(p^{\nu+1}n,p^{\rho}r,p^{\mu-1}mN\bigr)+p^{k-1} \alpha_f\bigl(p^{\nu-1}n,p^{\rho-1}r,p^{\mu-1}mN\bigr)  \\
 & -p^{k-1} \alpha_f\bigl(p^{\nu}n,p^{\rho-1}r,p^{\mu-2}mN\bigr).
\end{align*}
By induction hypothesis this equals
\begin{align*}
 & \sum^{\min(\nu+1,\rho,\mu-1)}_{\delta=0} p^{\delta(k-1)} \alpha_f\bigl(p^{\nu+\mu-2\rho}n,p^{\rho-1}r,mN\bigr)     + \sum^{\min(\nu,\rho,\mu)}_{\delta=1} p^{\delta(k-1)} \alpha_f\bigl(p^{\nu+\mu-2\delta}n,p^{\rho-\delta}r,mN\bigr) \\
 & \quad -\sum^{\min(\nu+1,\rho,\mu-1)}_{\delta=1} p^{\delta(k-1)} \alpha_f\bigl(p^{\nu+\mu-2\delta}n,p^{\rho-\delta}r,mN\bigr) \\
 & = \sum^{\min(\nu,\mu,\rho)}_{\delta=0} p^{\delta(k-1)} \alpha_f\bigl(p^{\nu+\mu-2\delta}n,p^{\rho-\delta}r,mN\bigr).
\end{align*}
The remaining cases are dealt with in the same way. \\
(iii)\;$\Rightarrow$\;(i) Inserting (iii) we get 
\[
\alpha_f(p^{\nu}n,p^{\rho}r,p^{\mu}mN) - p^{k-1} \alpha_f(p^{\nu-1}n,p^{\rho-1}r,p^{\mu-1}mN) = \alpha_f(p^{\nu+\mu}n,p^{\rho}r,mN),
\]
which yields the claim.  \\
(i)\;$\Leftrightarrow$\;(iv) Equation (iv) is equal to (i) applied to 
\[
 Z\begin{bmatrix}
   1 & 0  \\  0 & \sqrt{p} 
  \end{bmatrix}
= \begin{pmatrix}
   \tau & z\sqrt{p}  \\  z\sqrt{p} & p\tau'
  \end{pmatrix}
\]
instead of $Z$ because of 
\begin{align*}
 & \Sigma_N^J\,\tfrac{1}{\sqrt{p}}\diag(p,p,1,1)\Sigma_N^J  \\
 & = \Sigma_N^J \,\tfrac{1}{\sqrt{p}}\diag(p,p,1,1)\,
      \dot\cup \underset{u\bmod{p}}{\dot\bigcup}  
   \Sigma_N^J \,\frac{1}{\sqrt{p}}\left(\begin{pmatrix}
       1 & u  \\  0 & p                                                    
      \end{pmatrix} 
      \times\begin{pmatrix}
       p & 0  \\  0 & 1                                                    
      \end{pmatrix}\right).   
 \end{align*}
Since the remaining double cosets consist of a single right coset, the claim follows.
\end{proof}

As all the elliptic Hecke operators $\Tcal_n$ are polynomials in $\Tcal_p$, $p\in\Pb$, we get
\begin{corollary} 
Given $f\in\Mcal_k(\Sigma_N)$, $N\in\N$ the following assertions are equivalent:
\begin{enumerate}[(i)]
 \item $f\in\Mcal^*_k(\Sigma_N)$.
 \item $f \underset{k}{\mid} \Tcal_n^{\uparrow} = f \underset{k}{\mid} \Tcal_n^{\downarrow}$ for all $n\in\N$.
 \item $f \underset{k}{\mid} \Tcal_p^{\uparrow} = f \underset{k}{\mid} \Tcal_p^{\downarrow}$ for all $p\in\Pb$.
\end{enumerate}
\end{corollary}

\begin{proof}
We get \eqref{gl_2} from the application of part (iii) in Lemma 1 to all primes $p$ dividing $m$.
\end{proof}

The main result of this section is

\begin{theorem}\label{theorem 1} 
A modular form $f\in \Mcal_k(\Sigma_N)$, $N\in\N$ belongs to the Maaß space $\Mcal^*_k(\Sigma_N)$ if and only if there exists a finite set $S\subset \Pb$ such that
\[
 f \underset{k}{\mid} \Tcal_p^{\uparrow} = f \underset{k}{\mid} \Tcal_p^{\downarrow} \;\; \text{for all}\;\; p\in\Pb,\;p\notin S.
\]
\end{theorem}

\begin{proof}
 ``\;$\Rightarrow$\;'' Use Corollary 1 with $S=\emptyset$.  \\
 ``\;$\Leftarrow$\;'' Let $M= N \prod_{q \in S} q$. Then Lemma 1 holds for all $p\in\Pb$, $p\nmid M$. If $f^M$ denotes the Maaß lift of $f_N$ we consider $g:=f-f^M$. It follows from Lemma 1 and \eqref{gl_2} that $g$ has a Fourier-Jacobi expansion of the form
 \[
  \sum^{\infty}_{m=0} g_{mN}(\tau,z)\, e^{2\pi imN\tau'}, 
 \]
where at least one prime divisor of $m$ belongs to $S$. Hence all the Fourier-Jacobi coefficients $g_{mN}$ vanish, whenever $m$ is coprime to $M$. Thus the direct analog of Lemma 3 in \cite{BBK} yields $g\equiv 0$, hence $f=f^M \in\Mcal^*_k(\Sigma_N)$.
\end{proof}

\section{The maximal normal extension $\mathbf{\widehat\Sigma_N}$} \label{section 2}

If $N$ is squarefree and $\nu$ denotes the number of prime divisors of $N$ we obtain a maximal normal extension of index $2^{\nu}$ of $\Sigma_N$ given by 
\[
\widehat\Sigma_N = \langle \Sigma_N,F_d;d\mid N\rangle \subset \Sp_2(\R),
\]
where 
\[
 F_d = \begin{pmatrix}
        V^{-1}_d & 0  \\  0 & V^{tr}_d
       \end{pmatrix}, \quad 
 V_d = \frac{1}{\sqrt{d}}\begin{pmatrix}
        \alpha d & \beta N  \\  \gamma & \delta d
       \end{pmatrix}, \quad       
\alpha,\beta,\gamma,\delta\in\Z,\;\; \alpha\delta d-\beta\gamma N/d = 1.
\]
Note that $F^2_d\in \Sigma_N$. Instead of $V_d$ we may choose an arbitrary element of
\[
 \Gamma^{\circ}(N) V_d = V_d\Gamma^{\circ}(N) = \left\{\frac{1}{\sqrt{d}} \begin{pmatrix}
                                                                           *d & *N  \\  * & *d
                                                                          \end{pmatrix}\in SL_2(\R);\;*\in\Z\right\},
\]                                                                         
where
\[
\Gamma^{\circ}(N) = \left\{\begin{pmatrix}
                          * & *N  \\  * & *
                         \end{pmatrix}\in SL_2(\Z);\;*\in \Z\right\}.
\]
From Gritsenko \cite{G2} we obtain

\begin{lemma}\label{lemma 2} 
Let $N\in \N$ be squarefree and $d\mid N$. Then the mapping
\[
 \Mcal^*_k(\Sigma_N) \to \Mcal^*_k(\Sigma_N), \quad f\mapsto f \underset{k}{\mid} F^{-1}_d (Z) = F(Z[V^{tr}_d ]),
\]
is an isomorphism.
\end{lemma}

\begin{proof}
The Fourier coefficient of $T$ in $f(Z[V^{tr}_d])$ is equal to the Fourier coefficient of $T[V^{-1}_d]$ of $f$, hence by $\alpha_f(n',r',m'N)$ given by
\begin{gather*}\tag{6}\label{gl_6}
\begin{pmatrix}
 n' \\ r' \\ m'
\end{pmatrix}
= \begin{pmatrix}
 d & -\gamma & \gamma^2N/d \\ -2N & \alpha d+\gamma N/d & -2\alpha\gamma N \\ N/d & -\alpha & \alpha^2 d
\end{pmatrix}
\begin{pmatrix}
 n \\ r \\ m
\end{pmatrix}.
\end{gather*}
As the latter matrix belongs to $SL_3(\Z)$ we get 
\[
 gcd(n',r',m') = gcd(n,r,m).
\]
Then the claim follows from \eqref{gl_2}.
\end{proof}

Next we have a closer look at eigenforms under $F_d$.

\begin{theorem} 
Let $f\in\Mcal^*_k(\Sigma_N)$, $d\mid N$, $N \in \N$ squarefree, $\varepsilon = \pm 1$. Then the following assertions are equivalent
\begin{enumerate}[(i)]
 \item $f(Z[V^{tr}_d]) = \varepsilon\,f(Z)$.
 \item $\alpha_f(T[V^{-1}_d]) = \varepsilon\,\alpha_f(T)$ for all $T$.
 \item $\alpha_f(n,r,N) = \varepsilon\, \alpha_f(n',r',N)$, whenever 
  \[
 4nN-r^2 = 4n' N-r'^2, \;\;r\equiv -r'\bmod{2d}, \;\;r\equiv r'\bmod{2N/d}.
 \]
\end{enumerate}
\end{theorem}

\begin{proof}
Clearly (i) and (ii) are equivalent. In view of \eqref{gl_2} it suffices to require (ii) for all matrices $T$ with $m=1$. By
\[
 \alpha_f\left(T\begin{bmatrix}
                 1 & 0  \\  u & 1
                \end{bmatrix}\right)
= \alpha_f (T)\;\; \text{for}\;\; u\in \Z
\]
we apply \eqref{gl_6} in order to get for $T=\left(\begin{smallmatrix}
                                                    n & r/2 \\  r/2 & N
                                                   \end{smallmatrix}\right)$
\[
 \alpha_f\left(T[V^{-1}_d]\right) = \alpha_f\bigl(n'',(\alpha d+\gamma N/d)r,N\bigr)
\]
where $4n'' N-(\alpha d+\gamma N/d)^2 r^2 = 4nN-r^2$. In view of 
\[
 \alpha d+\gamma N/d = 1+ 2\gamma N/d = 2\alpha d-1
\]
we conclude
\begin{gather*}\tag{7}\label{gl_7}
(\alpha d+\gamma N/d)r\equiv \begin{cases}
                             -r & \mod{2d}\,, \\
                             r & \mod{2N/d}\,.
                            \end{cases}
\end{gather*}
Thus (ii) becomes equivalent with (iii).
\end{proof}

If we have $\varepsilon = 1$ for all $d$ the consequence is

\begin{corollary} 
A modular form $f\in \Mcal^*_k(\Sigma_N)$, $N\in \N$ squarefree, belongs to $\Mcal_k(\widehat\Sigma_N)$ if and only if there exists a function $\alpha^*_f:\N_0\to \C$ such that for all $T\geq 0$, $T\neq 0$
\[
 \alpha_f(T) = \sum_{\delta\mid gcd(n,r,m)} \delta^{k-1} \alpha^*_f \bigl((4N\det T)/\delta^2\bigr).
\]
\end{corollary}

\begin{proof}
It suffices to show that for $f\in \Mcal^*_k(\Sigma_N)\cap \Mcal_k(\widehat\Sigma_N)$
\[
 \alpha_f(n,r,N) = \alpha_f(n',r',N)
\]
holds, whenever
\[
 4nN -r^2 = 4n'N-r'^2.
\]
Elementary number theory shows that $r^2\equiv r'^2 \!\! \mod{4N}$ implies
\begin{align*}
 r & \equiv -r'\bmod{p}\quad \text{or}\quad r\equiv r'\bmod{p}\quad \text{for alle primes $p\mid N$ and $p=2$}  \\
 r & \equiv -r'\bmod{4}\quad \text{or}\quad r\equiv r'\bmod{4}\quad \text{if $N$ is even}.
\end{align*}
Thus we define
\begin{align*}
& d:=\prod_{\substack{p\mid N,\,p\,odd \\ r\equiv -r'\bmod{p}}}p\quad \text{if $N$ is odd or $N$ is even and $r\equiv r'\bmod{4}$},  \\
& d:=2\prod_{\substack{p\mid N,\,p\,odd \\ r\equiv -r'\bmod{p}}}p\quad \text{if $N$ is even and $r\equiv -r'\bmod{4}$}.
\end{align*}
Thus we get
\[
 (\alpha d + \gamma t/d) r\equiv r'\bmod{2N}
\]
in \eqref{gl_7}. Now the claim follows from Theorem 2.
\end{proof}

We give  a characterization of cusp forms in $\Mcal_k(\widehat{\Sigma}_N)$, i.e. only positive definite $T$ appear in the Fourier expansion. If $\phi$ denotes the Siegel $\phi$-operator we get 

\begin{lemma}\label{lemma 3} 
Given a squarefree $N\in\N$ and $f\in \Mcal_k(\widehat{\Sigma}_N)$ the following assertions are equivalent:
\begin{enumerate}[(i)]
\item $f$ is a cusp form.
\item $\alpha_f(n,0,0)=0$ for all $n\in\N_0$.
\item $f\mid \phi\equiv 0$.
\end{enumerate}
\end{lemma}

\begin{proof}
Given $T=\left(\begin{smallmatrix}
n & r/2  \\  r/2 & mN
\end{smallmatrix}\right)\geq 0$ 
with $\det T = 0$ and $T\neq 0$ there exists $U\in \Gamma^{\circ}(N) V_d$ for some $d\mid N$ such that
\[
 T[U] = \begin{pmatrix}
         t & 0  \\  0 & 0
        \end{pmatrix}, \quad t=gcd(n,r,m).
 \vspace*{-6ex} \\      
\]
\end{proof}
\vspace{2ex}
\begin{remarkson}
a) As we may choose $V_N = \left(\begin{smallmatrix}
                                   0 & N  \\  -1 & 0
                                  \end{smallmatrix}\right)$, 
we conclude
\[
 f \underset{k}{\mid} W_N = (-1)^k f
\]
for all $f\in \Mcal^*_k(\Sigma_N)$ in view of $\alpha_f(n,r,mM) =\alpha_f (m,r,nN)$ in \eqref{gl_2}.  \\
b) $\Sigma_N$ corresponds to the discriminant kernel in $S\Ocal(2,3)$ (cf. \cite{K1}). Moreover the extension $\Sigma_N\cup \Sigma_N W_N$ corresponds to the discriminant kernel in $\Ocal(2,3)$. \\
c) With $V_N = \left(\begin{smallmatrix}
               0 & N  \\  -1 & 0
              \end{smallmatrix}\right)$ one easily verifies
\[
 W_N(M\times M')W_N = M'\times M
\]
for $M,M'\in SL_2(\R)$ in \eqref{gl_4}. Thus one gets
\[
 f \underset{k}{\mid} \Tcal^{\downarrow}_n = f \underset{k}{\mid} W_N \underset{k}{\mid}  \Tcal^{\uparrow}_n \underset{k}{\mid} W_N.
\]
\end{remarkson}

\section{Hecke operators}

Consider the subgroup
\[
 \G:=\left\{\frac{1}{\sqrt{n}}M\in\Sp_2(\R);\,n\in\N,\,M\in\Z^{4\times 4}\right\}
\]
of $\Sp_2(\R)$ as well as its Jacobi subgroup
\[
 \G^J:= \left\{\begin{pmatrix}
                * & 0 & * & * \\ * & * & * & * \\ * & 0 & * & * \\ 0 & 0 & 0 & *
               \end{pmatrix}
\in\G\right\}.
\]
Then $(\Sigma_N,\G)$, $(\widehat{\Sigma}_N,\G)$ and $(\Sigma^J_N,\G^J)$ are Hecke pairs (cf. \cite{A}, \cite{F}, \cite{K2}). Denote the associated Hecke algebras by
\[
 \Hcal_N:=\Hcal(\Sigma_N,\G),\quad \widehat{\Hcal}_N:=\Hcal(\widehat{\Sigma}_N,\G), \quad \Hcal^J_N: = \Hcal(\Sigma^J_N,\G^J).
\]
The structure of $\Hcal_N$ and $\widehat{\Hcal}_N$ was described in \cite{GK} and moreover the 
\[
 \widehat{\Sigma}_N \cap \G^J = \widehat{\Sigma}_N, \quad \G\subset \widehat{\Sigma}_N \cdot \G^J.
\]
Thus the mapping
\[
 \sum_{M:\widehat{\Sigma}_N\backslash \G} t(\widehat{\Sigma}_N M)\widehat{\Sigma}_N M\mapsto 
 \sum_{M:\Sigma^J_N\backslash \G^J} t(\widehat{\Sigma}_N M)\Sigma^J_N M
\]
induces an injective homomorphism of the Hecke algebras 
\begin{gather*}\tag{8}\label{gl_8}
 \iota:\widehat{\Hcal}_N \to \Hcal^J_N
\end{gather*}
(cf. \cite{A}, Proposition 3.1.6, \cite{K2}, I(6.4)).

\begin{theorem}\label{theorem 3} 
Let $N\in\N$ be squarefree. Then the Maaß space $\Mcal^*_k(\Sigma_N)$ is invariant under the Hecke algebra $\Hcal_N$.
\end{theorem}

\begin{proof}
According to Lemma \ref{lemma 2} it suffices to demonstrate the invariance under $\widehat{\Hcal}_N$. It was shown in \cite{GK} that $\widehat{\Hcal}_N$ is generated by 
\[
 T_N(q):= \widehat{\Sigma}_N \tfrac{1}{\sqrt{q}} \diag(1,1,q,q)\widehat{\Sigma}_N \quad \text{and}\quad T^*_N(q):= \widehat{\Sigma}_N \tfrac{1}{q} \diag(1,q,q^2,q)\widehat{\Sigma}_N, \;\;q\in\Pb.
\]
According to \cite{GK} representatives of the right cosets in $T_N(q)$ are given by
\begin{gather*}\tag{9}\label{gl_9}
\begin{split}
 & \frac{1}{\sqrt{q}}\begin{pmatrix}
                    q & 0 & 0 & 0  \\  0 & q & 0 & 0  \\  & & 1 & 0  \\  & & 0 & 1
                   \end{pmatrix},\;
\frac{1}{\sqrt{q}}\begin{pmatrix}
1 & 0 & a & 0  \\  0 & q & 0 & 0 \\  & & q & 0  \\  & & 0 & 1
\end{pmatrix} ,  \\
& \frac{1}{\sqrt{q}}\begin{pmatrix}
1 & 0 & a & b  \\  0 & 1 & b & c/N  \\  & & q & 0  \\  & & 0 & q
\end{pmatrix},\;
\frac{1}{\sqrt{q}}\begin{pmatrix}
q & 0 & 0 & 0  \\  -d & 1 & 0 & c/N \\  & & 1 & d  \\  & & 0 & q
\end{pmatrix}
\end{split}
\end{gather*}
for $a,b,c,d \bmod{q}$, if $q\nmid N$. This means
\[
 \iota(T_N(q))=\Sigma^J_N\,\frac{1}{\sqrt{q}} \diag(1,q,q,1) \Sigma^J_N + \Sigma^J_N\, \frac{1}{\sqrt{q}} \diag(1,1,q,q)\Sigma^J_N.
\]
If $q\mid N$ additionally by
\[
 \frac{1}{q} \begin{pmatrix}
              q & 0 & 0 & b  \\  -d & q & b & 0  \\  & & q & d  \\  & & 0 & q
             \end{pmatrix} \in\Sigma^J_N \,\frac{1}{q}
             \begin{pmatrix}
              q & 0 & 0 & 1  \\  0 & q & 1 & 0  \\  & & q & 0  \\  & & 0 & q
             \end{pmatrix} \Sigma^J_N
\]
for $b,d \bmod{q}$, $(b,d)\not\equiv(0,0)\bmod{q}$ appear. Representatives of the right cosets in $T^*_N(q)$ are given by
\begin{gather*}\tag{10}\label{gl_10}
\begin{split}
 & \frac{1}{q}\begin{pmatrix}
                    q & 0 & 0 & 0  \\  0 & q^2 & 0 & 0  \\  & & q & 0  \\  & & 0 & 1
                   \end{pmatrix},\;
\frac{1}{q}\begin{pmatrix}
q^2 & 0 & 0 & 0  \\  -qd & q & 0 & 0  \\  & & 1 & d  \\  & & 0 & q
\end{pmatrix},\;
\frac{1}{q}\begin{pmatrix}
1 & 0 & a & b  \\  0 & q & qb & 0 \\  & & q^2 & 0  \\  & & 0 & q
\end{pmatrix}, \\
& \frac{1}{q}\begin{pmatrix}
q & 0 & u & ub  \\  0 & q & ub & ub^2  \\  & & q & 0  \\  & & 0 & q
\end{pmatrix},\; \frac{1}{q}\begin{pmatrix}
q & 0 & 0 & qb  \\  -d & 1 & b & c/N \\  & & q & qd  \\  & & 0 & q^2
\end{pmatrix},\; \frac{1}{q} \begin{pmatrix}
 q & 0 & 0 & 0  \\  0 & q & 0 & u/N  \\  & & q & 0  \\  & & 0 & q
\end{pmatrix}
\end{split}
\end{gather*}
for $b,d,u \bmod{q}$, $u\not\equiv 0\bmod{q}$ and $a,c \bmod{q^2}$. This means
\begin{align*}
 \iota(T^*_N(q)) = & \Sigma^J_N \,\tfrac{1}{q}\diag(q,q^2,q,1)\,\Sigma^J_N + \Sigma^J_N \,\tfrac{1}{q}\diag(1,q,q^2,q)\,\Sigma^J_N \\
		   & + \Sigma^J_N \,\tfrac{1}{q}\diag(q,1,q,q^2)\,\Sigma^J_N + \sum_{\substack{u\bmod q\\u\,\not\equiv\, 0\bmod q}} \Sigma^J_N 
		   \begin{pmatrix}
		    1 & 0  \\  0 & 1
		   \end{pmatrix} \times \begin{pmatrix}
					1 & u/q  \\  0 & 1
					\end{pmatrix} \Sigma^J_N.
\end{align*}
If $q\mid N$ additionally the right cosets
\begin{align*}
& \frac{1}{q}\begin{pmatrix}
q & 0 & 0 & b  \\  0 & q & b & c/N \\  & & q & 0  \\  & & 0 & q
\end{pmatrix},\quad 
\begin{matrix}
 b,c \bmod{q},\hspace*{2em} \\ b,c \not\equiv 0 \bmod{q},\;
\end{matrix}     \\
& \frac{1}{q}\begin{pmatrix}
q & 0 & 0 & b  \\  -d & q & b & c/N \\  & & q & d  \\  & & 0 & q
\end{pmatrix},\quad 
\begin{matrix}
 b,c,d \bmod{q},\;\;  \\ cd \not\equiv 0 \bmod{q},
\end{matrix}  \\
& \frac{1}{q\sqrt{q}}\begin{pmatrix}
q^2 & 0 & 0 & 0  \\  -qd & q^2 & 0 & 0 \\  & & q & d  \\  & & 0 & q
\end{pmatrix},\;
\frac{1}{q\sqrt{q}}\begin{pmatrix}
q & 0 & qa & b  \\  0 & q^2 & qb & 0  \\  & & q^2 & 0  \\  & & 0 & q
\end{pmatrix},\quad 
\begin{matrix}
 a,b,d \bmod{q},\;\; \\ bd \not\equiv 0 \bmod{q},
\end{matrix}   \\
& \frac{1}{q\sqrt{q}}\begin{pmatrix}
q^2 & 0 & 0 & qb  \\  -qd & q & b & bd+qc/N  \\  & & q & qd  \\  & & 0 & q^2
\end{pmatrix},\quad 
\begin{matrix}
 b,c,d \bmod{q}, \\ b \not\equiv 0 \bmod{q},
\end{matrix}  \\
& \frac{1}{q\sqrt{q}}\begin{pmatrix}
q & 0 & qa & qb  \\  -d & q & -ad+qb & bd+qc/N  \\  & & q^2 & qd  \\  & & 0 & q^2
\end{pmatrix},\quad 
\begin{matrix}
 a,b,c,d \bmod{q}, \\ d \not\equiv 0 \bmod{q}.\hspace*{1em} 
\end{matrix} 
\end{align*}
If $q\mid N$ therefore $\iota(T^*_n(q))$ equals
\begin{align*}
 & \Sigma^J_N \, \frac{1}{q} \diag(q,q^2,q,1)\,\Sigma^J_N + \Sigma^J_N \, \frac{1}{q} \diag(1,q,q^2,q)\,\Sigma^J_N + 
 \Sigma^J_N \, \frac{1}{q} \diag(q,1,q,q^2)\,\Sigma^J_N  \\[1ex]
 & + \sum_{\substack{u\bmod q\\u\,\not\equiv\, 0\bmod q}} \left[\Sigma^J_N \, \frac{1}{q}
			\begin{pmatrix}
                         q & 0 & 0 & 0  \\  0 & q & 0 & u/N  \\  & & q & 0  \\  & & 0 & q
                        \end{pmatrix}\Sigma^J_N 
 + \Sigma^J_N \, \frac{1}{q\sqrt{q}}
			\begin{pmatrix}
                         q^2 & 0 & 0 & q  \\  0 & q & 1 & qu/N  \\  & & q & 0  \\  & & 0 & q^2
                        \end{pmatrix}\Sigma^J_N \right] \\[1ex]
 & + \Sigma^J_N \, \frac{1}{q\sqrt{q}}
			\begin{pmatrix}
                         q & 0 & 0 & 1  \\  0 & q^2 & q & 0  \\  & & q^2 & 0  \\  & & 0 & q
                        \end{pmatrix}\Sigma^J_N  .                    
\end{align*}
A tedious straightforward calculation shows that $\iota(T_N(q))$ and $\iota(T^*_N(q))$ commute with
\[
 \Sigma^J_N \tfrac{1}{\sqrt{p}}\diag(p,p,1,1)\Sigma^J_N, \quad \Sigma^J_N \tfrac{1}{p} \diag(p,p^2,p,1) \Sigma^J_N
\]
and 
\[
 \sum_{u\bmod{p}} \Sigma^J_N \, I\times \begin{pmatrix}
                                         1 & u/p  \\  0 & 1
                                        \end{pmatrix}
 \Sigma^J_N,
\]
as long as $p\neq q$. Thus the claim follows from Theorem \ref{theorem 1}.
\end{proof}

Now we apply the Hecke operator $T_N(q)$ to the Fourier Jacobi expansion. 

\begin{corollary}\label{corollary 3} 
Let $f\in\Mcal^*_k (\Sigma_N)$ for squarefree $N\in\N$ and $q\in\Pb$, $q\nmid N$. Then 
\[
g:=f\underset{k}{\mid} T_N(q) \in \Mcal^*_k(\Sigma_N)
\]
is characterized by
\[
 g^*_N = f^*_N \underset{k}{\mid} \Sigma^J_N\,\tfrac{1}{q} \diag(1,q,q^2,q) \Sigma^J_N + (q^2+q) f^*_N.
\]
\end{corollary}

\begin{proof}
It follows from Gritsenko's work \cite{G1} that 
\[
 f^*_{qN} = \tfrac{1}{q} f^*_N \underset{k}{\mid} \Sigma^J_N \,\tfrac{1}{\sqrt{q}} \diag(q,q,1,1)\Sigma^J_N.
\]
Thus the application of \eqref{gl_9} leads to
\begin{align*}
g^*N & = f^*_{qN} \underset{k}{\mid} \Sigma^J_N\,\tfrac{1}{\sqrt{q}} \diag(1,1,q,q)\Sigma^J_N  \\
 & = \tfrac{1}{q} f^*_N \underset{k}{\mid} \Sigma^J_N\,\tfrac{1}{\sqrt{q}}\diag(q,q,1,1)\Sigma^J_N \cdot \Sigma^J_N\,\tfrac{1}{\sqrt{q}} \diag(1,1,q,q)\Sigma^J_N .
\end{align*}
If we use the representatives in \eqref{gl_9} and \eqref{gl_10}, the calculation of the product leads to the result.
\end{proof}

Note that $\widehat{\Sigma}_N M \widehat{\Sigma}_N = \widehat{\Sigma}_N M^{-1} \widehat{\Sigma}_N$ for all $M\in\G$ due to \cite{GK}. Hence the same arguments as in \cite{F}, IV.4.11, show that the Hecke operators are self-adjoint with respect to the Petersson scalar product. As $\widehat{\Hcal}_N$ is commutative due to \cite{GK}, the standard procedure yields

\begin{corollary}\label{corollary 4} 
Let $N\in\N$ be squarefree. Then the space $\Mcal^*_k(\Sigma_N)\cap \Mcal_k(\widehat{\Sigma}_N)$ contains a basis consisting of simultaneous eigenforms under all Hecke operators in $\widehat{\Hcal}_N$.
\end{corollary}

\section{The Eisenstein series}

Let
\[
 \Sigma_{N,\infty}:= \left\{\begin{pmatrix}
                             A & B  \\  0 & D
                            \end{pmatrix}\in\Sigma_N\right\}
\]
and define for even $k>3$ the Siegel-Eisenstein series
\[
 E_{k,N}(Z):= \sum_{M:\Sigma_{N,\infty}\backslash \Sigma_N} \det(CZ+D)^{-k}.
\]
We have
\[
 W^{-1}_d \Sigma_{N,\infty}W_d = \Sigma_{N,\infty} \quad \text{and}\quad W^{-1}_d\Sigma_N W_d = \Sigma_N
\]
and therefore $E_{k,N}\underset{k}{\mid} W_d = E_{k,N}$ for all $d\mid N$. Hence  $E_{k,N}\in\Mcal_k(\widehat{\Sigma}_N)$ holds. We use an analog of Elstrodt's result \cite{E} in order to obtain

\begin{theorem}\label{theorem 4} 
 Let $N\in\N$ be squarefree and $k>3$ be even. Then
 \[
  E_{k,N} \in \Mcal^*_k (\Sigma_N)\cap \Mcal_k(\widehat{\Sigma}_N).
 \]
\end{theorem}

\begin{proof}
As $J_{k,N}$ contains a non-cusp form, we conclude the existence of an $f\in\Mcal^*_k (\Sigma_N)$ with non-zero constant Fourier coefficient from \cite{G1}. Due to Lemma \ref{lemma 2} we may assume $f\in\Mcal_k(\widehat{\Sigma}_N)$ if we replace $f$ by $\sum\limits_{d\mid N} f \underset{k}{\mid} W_d$. By Corollary \ref{corollary 4} we may even assume 
\[
 \alpha_f(0,0,0) = 1 \;\;\text{and}\;\; f \underset{k}{\mid} \widehat{\Sigma}_N \,\tfrac{1}{\sqrt{q}} \diag(1,1,q,q) \widehat{\Sigma}_N = \lambda f
\]
for some prime $q\nmid N$ and eigenvalue $\lambda$. As $E_{k,N}$ is an eigenform, too, (cf. \cite{F}, IV.4.7) with the same eigenvalue (cf. \cite{F}, IV.4.6)
\[
 \lambda = q^k + q^2 + q + q^{3-k},
\]
this is also true for $g:=f-E_{k,N}$. Using Lemma \ref{lemma 1} we conclude that $f\mid \phi$ and $E_{k,N}\mid \phi$ coincide with the normalized elliptic Eisenstein series.  

Hence $g$ is a cusp form by Lemma \ref{lemma 3}. If $g\neq 0$, we get (cf. \cite{F}, IV.4.8) 
\[
 |\lambda| \leq \sharp\left(\Sigma_N\backslash \Sigma_N \tfrac{1}{\sqrt{q}} \diag(1,1,q,q)\Sigma_N\right) = 1+q+q^2+q^3 < q^4.
\]
In view of $\lambda \geq q^k\geq q^4$ we obtain a contradiction and therefore
\[
 f = E_{k,N}.
 \vspace*{-5ex}
\]
\end{proof}

If $k\geq 4$ is even we define the Jacobi Eisenstein series by
\[
e^*_{k,N}(Z):= \tfrac{1}{4} \sum_{M:\Gamma_N \backslash \Sigma^J_N} \phi_N \underset{k}{\mid} M, \quad 
\phi_N(Z):= e^{2\pi\,iN\tau'},
\]
where
\[
 \Gamma_N:=\left\{\begin{pmatrix}
                   I & \ast  \\  0 & I
                  \end{pmatrix} \in \Sigma_N\right\}
\]
is the translation subgroup of $\Sigma^J_N$. Note that
\[
 \alpha_{e^*_{k,N}} (0,0,N) = 1.
\]

\begin{corollary}\label{corollary 5} 
Let $k\geq 4$ be even, $N\in\N$ squarefree, $q\in\Pb$, $q\nmid N$. Then 
\[
 e^*_{k,N} \underset{k}{\mid} \Sigma^J_N \,\tfrac{1}{q} \diag(1,q,q^2,q)\Sigma^J_N = (q^k + q^{3-k}) e^*_{k,N}.
\]
\end{corollary}

\begin{proof}
By construction $e^*_{k,N}$ is invariant under $\Sigma^J_N$. If $K_i$, $i\in I$, denotes a set of representatives of $\Sigma^J_n\backslash \Sigma^J_N\,\tfrac{1}{q} \diag(1,q,q^2,q)\Sigma^J_N$ contained in \eqref{gl_10} and $M_j$, $j\in J$, of $\Gamma_N\backslash \Sigma^J_N$, then
\[
 e^*_{k,N} \underset{k}{\mid} \Sigma^J_N\,\tfrac{1}{q}\diag(1,q,q^2,q)\Sigma^J_N = \sum_{L:\Gamma_N\backslash \Sigma^J_N\,\frac{1}{q}\diag(1,q,q^2,q)\Sigma^J_N} \phi_N \underset{k}{\mid} L.
\]
Switching to the inverse we see that each $L\in\Sigma^J_N\,\frac{1}{q}\diag (1,q,q^2,q)\Sigma^J_N$ has a unique decomposition
\[
L = \pm K_i^{-1} N M_j,\;\; i\in I,\;j\in J,\; N\in \Gamma_N.
\]
Now observe that
\[
 \left[\frac{1}{q}\begin{pmatrix}
                   q^2 & 0 & 0 & 0  \\ -qd & q & 0 & 0  \\  & & 1 & d  \\  & & 0 & q
                  \end{pmatrix}\right]^{-1}
\Gamma_N = \underset{\substack{a\bmod{q^2} \\ b\bmod{q}}}{\dot\bigcup} \Gamma_N U_d\,\frac{1}{q}
\begin{pmatrix}
1 & 0 & a & b  \\ 0 & q & qb & 0  \\  & & q^2 & 0  \\  & & 0 & q
\end{pmatrix}, \quad U_d = 
\begin{pmatrix}
1 & 0 & 0 & 0  \\ d & 1 & 0 & 0  \\  & & 1 & -d  \\  & & 0 & 1
\end{pmatrix},
\]
\[
\sum_{\substack{a\bmod{q^2} \\ b\bmod{q}}}\phi_N \underset{k}{\mid} U_d \,\frac{1}{q} 
\begin{pmatrix}
1 & 0 & a & b  \\ 0 & q & qb & 0  \\  & & q^2 & 0  \\  & & 0 & q
\end{pmatrix} = \begin{cases}
                 q^{3-k} \phi_N, & \text{if}\;\, d\equiv 0\bmod{q}, \\
                 0, & \text{else}.
                \end{cases}
\]
\begin{align*}
\underset{\substack{a\bmod{q^2} \\ b\bmod{q}}}{\dot\bigcup}\left[\frac{1}{q}
\begin{pmatrix}
1 & 0 & a & b  \\ 0 & q & qb & 0  \\  & & q^2 & 0  \\  & & 0 & q
\end{pmatrix}\right]^{-1} \Gamma_N
 & = \Gamma_N \,\frac{1}{q}\begin{pmatrix}
                   q^2 & 0 & 0 & 0  \\ 0 & q & 0 & 0  \\  & & 1 & 0  \\  & & 0 & q
                  \end{pmatrix},    \\
\phi_N \underset{k}{\mid} \tfrac{1}{q} \diag(q^2,q,1,q)  & = q^k \phi_N,   \\
\left[\frac{1}{q}
\begin{pmatrix}
q & 0 & u & ub  \\ 0 & q & ub & ub^2  \\  & & q & 0  \\  & & 0 & q
\end{pmatrix}\right]^{-1} \Gamma_N  & = \Gamma_N \,\frac{1}{q}
\begin{pmatrix}
q & 0 & -u & -ub  \\ 0 & q & -ub & -ub^2  \\  & & q & 0  \\  & & 0 & q
\end{pmatrix},    \\
\sum_{\substack{u,b\bmod{q} \\ u \,\not\equiv\, 0\bmod{q}}}\phi_N \underset{k}{\mid} \frac{1}{q} 
\begin{pmatrix}
q & 0 & -u & -ub  \\ 0 & q & -ub & -ub^2  \\  & & q & 0  \\  & & 0 & q  
\end{pmatrix} & = 0.
\end{align*}
Hence the claim follows.
\end{proof}

Finally we derive a connection with $E_{k,N}$ and the classical Siegel Eisenstein series $E_{k,1}$.               

\begin{corollary}\label{corollary 6} 
Let $k\geq 4$ be even and $N\in\N$ squarefree.
\begin{enumerate}[a)]
\item The $N$-th Fourier Jacobi coefficient of $E_{k,N}$ is given by 
\[
 -\frac{B_k}{2k} \, e^*_{k,N}.
\]
\item One has 
\[
 \sigma_{k-1} (N)\cdot e^*_{k,N} = \tfrac{1}{N} \, e^*_{k,1} \underset{k}{\mid} \Sigma^J_1\,\tfrac{1}{\sqrt{N}} \diag(1,N,N,1)\Sigma^J_1.
\]
\end{enumerate}
\end{corollary}

\begin{proof}
a) Combine Corollary \ref{corollary 4}, Theorem \ref{theorem 4} and \eqref{gl_3}.  \\
b) Note that $\Sigma^J_N\,\frac{1}{\sqrt{N}} \diag(1,N,N,1)\Sigma^J_N$ and $\Sigma^J_N\,\frac{1}{q} \diag(1,q,q^2,q)\Sigma^J_N$ commute for $q\in\Pb$, $q\nmid N$ and use the representatives above in order to compute the Fourier coefficient of 
$\left(\begin{smallmatrix}
   0 & 0  \\  0 & N 
\end{smallmatrix}\right) $ on the right hand side.                                                                                                                                                                                                      
\end{proof}

Hence the Fourier coefficients of $E_{k,N}$ can be found in \cite{EZ}, \S 2.

\vspace{6ex}

 \nocite{H}  


\bibliography{bibliography_krieg} 
\bibliographystyle{plain}
\end{document}